\documentclass{amsart}

\usepackage{amssymb}
\usepackage{mathtools}

\newtheorem{theorem}{Theorem}[section]
\newtheorem{lemma}[theorem]{Lemma}
\newtheorem{proposition}[theorem]{Proposition}

\theoremstyle{definition}
\newtheorem{definition}[theorem]{Definition}

\theoremstyle{remark}
\newtheorem{remark}[theorem]{Remark}

\numberwithin{equation}{section}

\usepackage[inline]{enumitem}
\usepackage{mathrsfs}

\newcommand{\halffrac}{{\frac{1}{2}}}
\newcommand{\alphafrac}{{\frac{\alpha}{2}}}
\newcommand{\mr}{\mathbb{R}}
\newcommand{\prob}{\mathbb{P}}

\begin{document}

\title[On the LIL for Brownian motion on compact manifolds]{On the law of the iterated logarithm for Brownian motion on compact manifolds}


\author{Cheng Ouyang}
\address{Department of Mathematics, Statistics, and Computer Science \\ University of Illinois at Chicago \\ 322 Science and Engineering Offices (M/C 249) \\ 851 S. Morgan Street \\ Chicago, Illinois 60607}
\email{couyang@math.uic.edu}
\thanks{The research of the first author is supported in part {by Simons} grant \#355480.}

\author{Jennifer Pajda-De La O}
\address{Department of Mathematics, Statistics, and Computer Science \\ University of Illinois at Chicago \\ 322 Science and Engineering Offices (M/C 249) \\ 851 S. Morgan Street \\ Chicago, Illinois 60607}
\email{jpajda2@uic.edu}
\thanks{}

\subjclass[2010]{Primary 60B10, 60D05}


\dedicatory{}

\commby{}

\begin{abstract}
By taking a functional analytic point of view, we consider a family of distributions  (continuous linear functionals on smooth functions), denoted by  $\{\mu_t,t>0\}$, associated to the law of iterated logarithm for Brownian motion on a compact manifold. We give a complete characterization of the collection of  limiting distributions of $\{\mu_t, t>0\}$.
\end{abstract}

\maketitle


\section{Introduction}

Let $M$ be a compact $C^\infty$-Riemannian manifold (without boundary). For any fixed $x\in M$, it is well-known that the Laplace-Beltrami operator $\triangle_M$ generates a unique diffusion process $X$ starting from $x$, which is called the Brownian motion on $M$ starting from $x$. It is a continuous, strong Markov process with transition density $p(t,x,y)$, the fundamental solution of 
$$\frac{\partial}{\partial t}p(t,x,\cdot)=\frac{1}{2}\triangle_Mp(t,x,\cdot)\,.$$

Denote by $m$ the volume measure on $M$ induced by the metric and $m_0=m(M)$. Since $m/m_0$ is the invariant probability measure for $X$, the well-known ergodic theorem implies that for all $f\in L^1(M)$, almost surely
\begin{align*}
\lim_{t \rightarrow \infty} \frac{1}{t} \int_0^t f(X_s) ds = \frac{1}{m_0} \int_M f dm.
\end{align*}
Hence, $\int_0^t f(X_s) \ ds$ blows up to infinity with a rate of $t$ and scalar $\frac{1}{m_0} \int_M f dm$.  The second order term is given by 
\[ \int_0^t f(X_s) ds - \frac{t}{m_0} \int_M f dm. \] 
When $t$ tends to infinity, the magnitude of the above is {characterized} by the law of the iterated {logarithm} (see, for example \cite{BB1976} and \cite{Brosamler}). More precisely, we have almost surely
\begin{align}\label{LIL}
\limsup_{t\to\infty}\frac{\int_0^tf(X_s)ds-m_0^{-1}t\int_Mfdm}{\sqrt{2t\log\log t}}=\sqrt{\frac{2}{m_0}(Gf,f)_{L^2}}.
\end{align}
In the above, $(\cdot,\cdot)_{L^2}$ is the standard inner product associated to $L^2(M)$, and $G$ is the Green operator introduced  by Baxter and Brosamler in \cite{BB1976} (see Section 2 below for an explicit definition of $G$). It took some effort to show that (\ref{LIL}) is true simultaneously for all $f\in C^\infty(M)$. Indeed, the following was proved by Brosamler in \cite{Brosamler}.
\begin{theorem}\label{LIL all}
We have
\begin{align}\label{LIL uniform}
\prob\bigg\{\limsup_{t\to\infty}&\frac{\int_0^tf(X_s)ds-m_0^{-1}t\int_Mfdm}{\sqrt{2t\log\log t}}\\
&\quad\quad\quad\quad\quad\quad\quad=\sqrt{\frac{2}{m_0}(Gf,f)_{L^2}}, \ \mathrm{all}\ f\in C^\infty\bigg\}=1.\nonumber
\end{align}
\end{theorem}

\bigskip
The present work is concerned with the family of signed measures $\{\mu_t, t>0\}$ on $M$ obtained by
\begin{align*}
\int_M fd\mu_t=\frac{\int_0^tf(X_s)ds-m_0^{-1}t\int_Mfdm}{\sqrt{2t\log\log t}}.
\end{align*}
With Theorem \ref{LIL all} in mind, we would rather think { of} $\mu_t$ as a distribution (continuous linear functional) on $C^{\infty}(M)$ and write, for all $f\in C^\infty(M)$
$$\mu_t(f)=\int_Mfd\mu_t.$$
Of {course,} $\mu_t$ also depends on the sample path $\omega\in\Omega$ of $X$. In the event that we want to emphasize such  dependence, we will write $\mu_t^\omega$. Throughout this paper, we use the term {\it distributions} exclusively for continuous linear functionals on $C^\infty(M)$. It is then natural to wonder {what the limiting distributions of the family $\{\mu_t\}$ are}. More precisely, we are interested in, for each $\omega\in\Omega$, how one can characterize the class of distributions $\mu$ on $C^\infty(M)$  such that there exists a sequence $\{t_n, n\geq1\}$
\begin{align}\label{testing}\mu_{t_n}^\omega(f)\to \mu(f),\quad\quad \mathrm{all}\ f\in C^\infty(M).\end{align}
The result of our investigation on the above question is reported in the next theorem.
\begin{theorem}\label{Main thm}
Let $\mathscr{D}$ be the collection of distributions on $C^\infty(M)$ satisfying the following four properties. 
\begin{itemize}
\item[(a)] $\mu$ can be identified as a signed measure on $M$, still denoted by $\mu$;
\item[(b)] $\mu(M)=0$;
\item[(c)] $\mu$ is absolutely continuous with respect to the volume measure $m$, with Radon-Nikodym derivative $g=d\mu/dm$  in $L^2(M)$. Moreover,
\item[(d)] $g$ is in the domain of $G_{1/2}^{-1}$, and $$\|G_{\frac{1}{2}}^{-1}g\|_{L^2}\leq \sqrt{\frac{2}{m_0}}.$$
\end{itemize}
Here, roughly speaking, $G_{1/2}=(-\triangle_M/2)^{-1/2}$, and is defined more precisely in the next section.

Then, almost surely, the class of limiting distributions of the family $\{\mu_t^\omega,t>0\}$ is exactly $\mathscr{D}$. 
\end{theorem}
Clearly, the above theorem gives a complete characterization of the collection of limiting distributions of $\{\mu_t\}$.

Careful readers may wonder whether {Theorem} \ref{Main thm} remains valid if we replace $C^\infty(M)$ in (\ref{testing}) by $C(M)$, the collection of (bounded) continuous functions on $M$. That is, we regard $\{\mu_t\}$ and $\mu$ in Theorem \ref{Main thm} as genuine signed measures and consider {weak} convergence in the space of signed measures, as opposed {to convergence} in the space of distributions.  {The main reason why we do not work in the former setting in the present work is due to the fact that, in that case, Theorem \ref{Main thm} is intimately related to a version of {the law of the iterated logarithm} that holds simultaneously for all {$f\in C(M)$}}. Whether one has such an iterated logarithm is a non-trivial question. It seems the best result {in} this direction is obtained by Brosamler in \cite{Brosamler} for the Sobolev spaces $H^\alpha_0$ with $\alpha>\max(d-3/2, d/2)$, where $d$ is the dimension of $M$ and $H_0^\alpha$ is that in Definition \ref{H space} below.

\bigskip

The rest of this paper has two sections. In Section 2, we provide some preliminary material that will be needed for our discussion later. In particular, we introduce some operators and Sobolev spaces associated to Brownian motion on manifolds.
The proof of Theorem 1.2 is detailed in Section 3.

\section{Brownian motion on manifolds and Sobolev spaces}

Throughout our discussion below, we assume that $M$ is a compact Riemannian manifold of dimension $d$. Fix any $x\in M$, let $X=\{X_t, t\geq 0\}$ be a Brownian motion on $M$ starting from $x$; that is, $X$ is the unique diffusion process generated by the Laplace-Beltrami operator $\triangle_M$. Denote by $p(t,x,y)$ its  probability transition density function. In this section, we briefly introduce some operators and Sobolev spaces associated to $X$, that will be needed in the sequel. A more detailed discussion on these materials can be found, e.g., in \cite{Brosamler}.

We first introduce the Green kernel
\begin{align}\label{g}
g(x,y)=\int_0^\infty (p(t,x,y)-m_0^{-1})\,dt,\quad\quad x,y\in M, x\not=y.
\end{align}
Clearly $g(x,y)$ inherits its symmetry in $x$ and $y$ from $p(t,x,y)$.  It is also not hard to see that $g$ is continuous off the diagonal of $M\times M$.  Since for large $t$, there exists $\alpha>0$  and $C>0$ such that (see, e.g., \cite{BB1976})
$$\sup_{x,y\in M}|p(t,x,y)-m_0^{-1}|\leq Ce^{-\alpha t},$$
and for small $t$, $p(t,x,y)$  is known to have an order (see, e.g., \cite{M1953} and \cite{Hsu})
$$ (2\pi t)^{-\frac{d}{2}}e^{-\frac{d(x,y)^2}{2t}}.$$
Here $d(x,y)$ is the Riemannian distance between $x$ and $y$. Moreover, with some extra work, one can show that $g(x,y)$ is $C^\infty$ off the diagonal.

More generally, we introduce for $\alpha>0$ the kernel
\begin{align}\label{g alpha}
g_\alpha(x,y)=\Gamma(\alpha)^{-1}\int_0^\infty t^{\alpha-1}(p(t,x,y)-m_0^{-1})\,dt,\quad\quad x,y\in M, x\not=y.
\end{align}
Obviously $g_1(x,y)=g(x,y)$, and $g_\alpha(x,y)$ is symmetric in $x$ and $y$. The semigroup property of $p(t,x,y)$ implies that 
$$\int_Mg_\alpha(x,z)g_\beta(z,y)m(dz)=g_{\alpha+\beta}(x,y),\quad \mathrm{for}\ \alpha, \beta>0.$$
Hence, for any $f\in L^2(M)$, letting
$$(G_\alpha f)(x)=\int_Mg_\alpha(x,y)f(y)m(dy),$$
we obtain a semigroup of bounded symmetric linear operators $G_\alpha$ on $L^2(M)$. In particular, we denote $G=G_1$ which is the Green operator introduced in \cite{BB1976}.

In the following,  we let
$$L^2_0(M)=\left\{f\in L^2(M): \int_Mfdm=0\right\},$$
{and}
$$C^\infty_0(M)=\left\{f\in C^\infty(M): \int_Mfdm=0\right\}.$$
For simplicity, we will suppress $M$ in the notation when there is no danger of possible confusion.

\begin{definition}\label{H space}
For any $\alpha>0$, let $H_0^\alpha=G_{\alpha/2}(L^2_0)$, with inner product
\begin{align*}
{  \left( G_{\alphafrac}f_1, G_{\alphafrac}f_2 \right)_{H_0^\alpha}=2^\alpha(f_1, f_2)_{L^2}.   }
\end{align*}
We write $\|\cdot\|_{H_0^\alpha}$ for the norm induced by $(\cdot, \cdot)_{H_0^\alpha}$.
\end{definition}
It is known (see, e.g., \cite{Brosamler}) that for $\alpha_1<\alpha_2$, $H_0^{\alpha_2}$ is continuously embedded into $H_0^{\alpha_1}$, and $\cap_{\alpha>0}H_0^\alpha=C^\infty_0$. Moreover, the Sobolev spaces $H_0^\alpha$ are the completion of $C_0^\infty$ with respect to the norm $\|\cdot\|_{H_0^\alpha}$.

\smallskip
The characterization of operators $G_\alpha$ and spaces $H_0^\alpha$ is probably more familiar to some readers in a functional analytic setting.
Denote by $0=\lambda_0<\lambda_1\leq\lambda_2\leq...$ the eigenvalues of $-\triangle_M$ and by $\phi_0={m_0^{-1/2}}, \phi_1,\phi_2,...$ an orthonormal sequence of corresponding eigenfunctions.
\begin{proposition}\label{Fourier}
The following facts are well-known.
\begin{enumerate}
\item $\phi_n\in C^\infty$, $n\geq 0$.
\item $G_\alpha \phi_0=0$ and $$G_\alpha \phi_n=2^\alpha\lambda_n^{-\alpha}\phi_n,$$ for $\alpha>0$.
\item Set $$\phi_n^\alpha=\lambda_n^{-\alpha/2}\phi_n.$$ For all $\alpha>0$, the functions $\{\phi_n^\alpha, n\geq1\}$ form a complete orthonormal system in $H_0^\alpha$.
\item For any $f\in L^2_0$, let $f_n=(f,\phi_n)_{L^2}$. We have $$f=\sum_{n=1}^\infty f_n\phi_n.$$ Moreover, $f$ belongs to $H_0^\alpha$ if and only if $$\sum_{n=1}^\infty \lambda^\alpha_n f_n^2<\infty.$$
\item For $f\in H^\alpha_0$, $$\|f\|_{H_0^\alpha}^2=\sum_{n=1}^\infty \lambda_n^\alpha f_n^2.$$
\end{enumerate}
\end{proposition}

\begin{remark}
By the characterization in terms of eigenvalues and {eigenfunctions}, it is clear that $G_\alpha: L^2_0\to H_0^{2\alpha}$ is a self-adjoint operator and, indeed, $G_\alpha= (-\triangle_M/2)^{-\alpha}$.
\end{remark}

Finally, we state one of  the main results in \cite{Brosamler}, which plays a key role in our discussion below. 
\begin{theorem}\label{key conti est}
For $\alpha>\max(d-3/2, d/2)$, we have almost surely
$$\left|\frac{\int_0^t f(X_s)ds}{\sqrt{2t\log\log t}}\right|\leq \|f\|_{H_0^\alpha} C(\omega),\quad t\geq 3, f\in H_0^\alpha.$$
In the above, $C$ is a finite constant depending on sample path $\omega$ (but not on $t$).
\end{theorem}
\begin{proof}
This is essentially the content of \cite[Theorem 3.8]{Brosamler}.
\end{proof}

\section{Limiting distributions}
Recall that $X$ is a Brownian motion on a compact manifold $M$ starting from a pre-fixed point $x\in M$. For each $t>0$ and $\omega\in\Omega$, we consider the distribution $\mu_t^\omega$ on $C^\infty(M)$ obtained by
$$\mu_t^\omega(f)=\frac{\int_0^tf(X_s(\omega))ds-m_0^{-1} t\int_Mfdm}{\sqrt{2t\log\log t}},\quad\quad f\in C^\infty(M).$$
To lighten the notation, we usually suppress its dependence on $\omega$ and simply write $\mu_t$. We are interested in understanding the class of limiting distributions (accumulating points) of the family $\{\mu_t\}$. Clearly, in order to prove Theorem \ref{Main thm}, we only need to show the following two theorems hold.

\begin{theorem}\label{Thm 1}
For each  $\omega\in \Omega$, if $\mu$ is a limiting distribution of the family $\{\mu_t^\omega, t\geq 0\}$, then $\mu$ can be identified as a signed measure on $M$, still denoted by $\mu$, such that
\begin{itemize}
\item[(a)] $\mu(M)=0$;
\item[(b)] $\mu$ is absolutely continuous with respect to the volume measure $m$, with Radon-Nikodym derivative $g=d\mu/dm$  in $L^2_0(M)$. Moreover,
\item[(c)] $g$ is in the domain of $G_{1/2}^{-1}$, and \[ { \left\|G_\halffrac^{-1}g\right\|_{L^2}\leq \sqrt{\frac{2}{m_0}}.   }  \]
\end{itemize}
\end{theorem}

\begin{theorem}\label{Thm 2}
There exists a subset $\Omega_0\subset \Omega$ with $\prob(\Omega_0)=1$ such that for any signed measure $\mu$ satisfying the characterizations in Theorem \ref{Thm 1} and any $\omega\in\Omega_0$, we can find a sequence of times $\{t_n, n\geq1\}$ such that for all $f\in C^\infty(M)$,
$$\mu^\omega_{t_n}(f)\to\mu^\omega(f),$$
as $n\to \infty$.
\end{theorem}

\bigskip
The rest of this section is devoted to the proof of Theorem \ref{Thm 1} and Theorem \ref{Thm 2} above.

\bigskip

\noindent{\bf {Proof of Theorem \ref{Thm 1}.}} \quad Fix any $\omega\in\Omega$ in (\ref{LIL uniform}), and suppose $\mu$ is a limiting distribution of the family $\{\mu_t^\omega, t\geq 0\}$. First note that $\mu$ can be identified as (or, in another word, extended to) a signed measure on $M$. Indeed, by (\ref{LIL uniform}) and the fact that $G: L^2(M)\to L^2(M)$ is a bounded linear operator, we have for all $f\in C^\infty$,
$$|\mu(f)|\leq \sqrt{\frac{2}{m_0}(f, Gf)_{L^2}}\leq C\sqrt{\frac{2}{m_0}}\|f\|_{L^2}$$
for some constant $C>0$. Since $C^\infty(M)$ is dense in $L^2(M)$, the above inequality implies that $\mu$ can be extended to (and hence be identified as) a bounded linear functional on $L^2(M)$. {Now, the} Riesz representation theorem tells us that there exists a function $g\in L^2(M)$ such that
$$\mu(f)=(g, f)_{L^2}=\int_M fgdm,\quad\quad f\in L^2(M).$$
As a consequence, we can identify $\mu= g\,dm$, a signed measure on $M$ which is absolutely continuous with respect to the volume measure $m$. Clearly, the {Radon-Nikodym} derivative $g$ is  $L^2(M)$. In addition, for any constant function $f$, we have $$\mu_t(f)=0.$$ It  implies $$\mu(f)=0,$$ and in particular for $f\equiv1$, $$\mu(M)=\int_Mg\,dm=0.$$ Hence $g\in L^2_0(M)$.  This proves (a) and (b) of Theorem \ref{Thm 1}. 

Next, we show that $g$ satisfies (c) of Theorem \ref{Thm 1}. For any $f\in C_0^\infty$, denote by
$$h=G_{\frac{1}{2}}f.$$
By (\ref{LIL uniform}), we have
\begin{align*}
\left|\int_M fgdm\right|&=\left|\mu(f)\right|\\
&\leq\sqrt{\frac{2}{m_0}(f, Gf)_{L^2}}\\
&=\sqrt{\frac{2}{m_0}(G_{\frac{1}{2}}f,G_{\frac{1}{2}}f)_{L^2}}\\
&=\sqrt{\frac{2}{m_0}}\|h\|_{L^2}.
\end{align*}
That is
\begin{align}\label{dual argument}
\left|(G^{-1}_{\frac{1}{2}}h, g)_{L^2}\right|=\left|\int_M \left(G^{-1}_{\frac{1}{2}}h\right)\, gdm\right|\leq \sqrt{\frac{2}{m_0}}\|h\|_{L^2},
\end{align}
for all $h=G_{1/2}f, \,f\in C_0^\infty$.  Observe that $C_0^\infty=\cap_{\alpha\geq0}H_0^\alpha$, together with the definition of $H_0^\alpha$, we have $G_{1/2}(C^\infty_0)=C^\infty_0$. Thus we conclude that (\ref{dual argument}) holds valid for all $h\in C^\infty_0$. As a consequence, $g$ is in the domain of $(G^{-1}_{1/2})^*$, the adjoint of $G^{-1}_{1/2}$, for $C_0^\infty$ is dense in $L^2_0$. Thus
$$\left(h, \left(G_\frac{1}{2}^{-1}\right)^*g\right)_{L^2}\leq\sqrt{\frac{2}{m_0}}\|h\|_{L^2}.$$
Again, the density of  $C_0^\infty$ in $L_0^2$ and the above inequality implies
\[ {  \left\| \left(G^{-1}_\halffrac \right)^*g \right\|_{L^2}\leq\sqrt{\frac{2}{m_0}}.   }    \]
Now the proof of (c) is completed by observing that $G^{-1}_{1/2}$ is a self-adjoint operator on $L_0^2$\,.\hfill${\Box}$

\bigskip

Finally, we focus on the proof of Theorem \ref{Thm 2}.  {First,} observe that for $f_1,..., f_n\in L^2_0$, the matrix $((f_i, Gf_j)_{L^2}, i,j=1,...,n)$ is positive definite if and only if $f_1,...,f_n$ are linearly independent. Let $f_1,...,f_n\in L^2_0$ be linearly independent and consider the ellipsoid $E_{f_1,...,f_n}$ defined by
\begin{align}\label{ellipsoid}E_{f_1,...,f_n}=\left\{(z_1,...,z_n)\in\mr^n, \sum_{i,j=1}^na_{ij}z_iz_j\leq1\right\}.\end{align}
Here $(\frac{m_0}{2}a_{ij})$ is the inverse matrix of $((f_i, Gf_j)_{L^2}, i,j=1,...,n)$. 

\begin{remark}\label{fk}
Recall our $\phi_1,\phi_2,...$ in Proposition \ref{Fourier}. Clearly $\phi_1,\phi_2,...$ are linearly independent and $\phi_k\in C_0^\infty$ for all $k\geq1$. Since $$G\phi_k=2\lambda_k^{-1}\phi_k,$$ we have for $f_k=\sqrt{\lambda_k/2}\,\phi_k$,
$$(f_i, Gf_j)_{L^2}=\delta_{ij}.$$
{Throughout} our discussion below we pick this particular choice of $f_k$'s. In this case, for each $n\geq1$, $E_{f_1,...,f_n}$ is simply a ball in $\mr^n$ centered at the origin  with radius $\sqrt{2/m_0}$.
\end{remark}

\begin{lemma}\label{lemma cluster}
Suppose $\alpha>\max\left(d-\frac{3}{2},\frac{d}{2}\right)$, and denote by
$$L_t(f)=\int_0^tf(X_s)ds.$$
There exists a subset $\Omega_0\subset\Omega$ with $\prob(\Omega_0)=1$ such that for all $n\geq1$ 
\begin{align*}
\mr^n-\mathrm{ cluster\ set}\ \frac{(L_t(f_1),...,L_t(f_n))}{\sqrt{2t\log\log t}}=E_{f_1,...,f_n},\quad\quad \mathrm{when}\ t\to\infty.
\end{align*}
\end{lemma}
\begin{proof}
This a restatement of Theorem 4.6 of \cite{Brosamler}.
\end{proof}

\bigskip

\noindent{\bf Proof of Theorem \ref{Thm 2}.} \quad We want to show that for any  $\omega\in\Omega_0$ in Lemma \ref{lemma cluster} and any fixed $\mu$, a signed measure satisfying the characterizations in Theorem \ref{Thm 1}, we can find a sequence of times $t_1<t_2<t_3<...$ such that $$\mu_{t_n}\to\mu,\quad\quad\mathrm{as}\ \  n\to\infty,$$ in the space of distributions.  We break our proof into three steps. 
\bigskip

\noindent{\it Step 1.} We first show that for any $n\geq 1$, the vector $$v_n=(\mu(f_1),...,\mu(f_n))$$ is an element in the ball $E_{f_1,...,f_n}$ defined in (\ref{ellipsoid}).

By our choice of $\{f_k, k\geq1\}$ in Remark \ref{fk}, $E_{f_1,...,f_n}$ is a ball in $\mr^n$ centered at the origin with radius $\sqrt{2/m_0}$. Hence the proof reduces to show that the inner product $v_n\cdot\zeta$ in $\mr^n$ satisfies
\begin{align}\label{in ball}|v_n\cdot\zeta|\leq \sqrt{\frac{2}{m_0}},\end{align}
for any unit vector $\zeta=(\zeta^1,...,\zeta^n)\in\mr^n$.  Indeed, we have
\begin{align*}
v_n\cdot\zeta&=\sum_{k=1}^n \mu(f_k)\zeta^k\\
&=\mu\left(\sum_{k=1}^n \zeta^kf_k\right)\\
&=\int_M \left(\sum_{k=1}^n\zeta^kf_k\right) g\, dm\\
&=\left(\sum_{k=1}^n\zeta^kf_k\ ,\  g\right)_{L^2}\\
&=\left(\sum_{k=1}^n\zeta^kG_{\frac{1}{2}}f_k\ ,\  G_{\frac{1}{2}}^{-1}g\right)_{L^2},
\end{align*}
where $g=d\mu/dm$. {Hence, by the Cauchy-Schwarz inequality and our choice of $f_k$ and $\mu$, we obtain }
\begin{align*}
\|v_n\cdot\zeta| &\leq \left\|\sum_{k=1}^n\zeta^k G_{\frac{1}{2}}f_k\right\|_{L^2}\left\|G_{\frac{1}{2}}^{-1}g\right\|_{L^2}\\
&=\sqrt{(\zeta^1)^2+..+(\zeta^n)^2}\,\left\|G_{\frac{1}{2}}^{-1}g\right\|_{L^2}\\
&\leq \sqrt{\frac{2}{m_0}},
\end{align*}
where we used the fact that 
$$\left\|G_{\frac{1}{2}}^{-1}g\right\|_{L^2}\leq \sqrt{\frac{2}{m_0}}.$$
Hence we have proved the desired inequality in (\ref{in ball}).

\bigskip
\noindent{\it Step 2.}  Fix any $\omega\in\Omega_0$ in Lemma \ref{lemma cluster}. We show in this step that we can find a sequence of {times} $t_1<t_2<t_3<...$ such that
\begin{align}\label{convergence v_n}
\mu_{t_n}(f_k)\to \mu(f_k),\quad\quad\mathrm{as}\ n\to\infty,
\end{align}
for all $k\geq1$.

For each fixed $n\geq1$, still let $$v_n=(\mu(f_1),...,\mu(f_n)).$$ 
Recall that for $f\in C_0^\infty$,
$$\mu_t(f)=\frac{\int_0^tf(X_s)ds}{\sqrt{2t\log\log t}}=\frac{L_t(f)}{\sqrt{2t\log\log t}}.$$
Introduce
$$v_{n,t}=(\mu_t(f_1),...,\mu_t(f_n)).$$
By Lemma \ref{lemma cluster} and what we have proved in {\it Step 1},  for each fixed $n$ there exists an increasing sequence of {times} $\{t^n_m, m\geq1\}$ such that
$$|v_{n,t^n_m}-v_n|\to 0,\quad \quad\mathrm{as}\ m\to\infty.$$

Start with $n = 1$.  Because $v_{1,t_m^1} \rightarrow v_1$,  in particular, we can find $t_1 \in \{t_m^1\}$ such that  \[ |{v_{1,t_1} - v_1}| < 1. \]  
For $n = 2$,  $v_{2,t_m^2} \rightarrow v_2$ implies that  we can find $t_2 \in \{t_m^2\}$ such that $t_2>t_1$ and
\[ { |{v_{2,t_2} - v_2}| < \frac{1}{2}. }\]
In general, we can choose $t_n \in \{ t_m^n \}$ such that $t_n>t_{n-1}$ and
\begin{align}|{ v_{n,t_n} - v_n }| < \frac{1}{n}. \label{eq:21}\end{align}

We claim that with such choice of $\{t_n, n\geq1\}$,  the convergence in (\ref{convergence v_n}) holds true for all $k$. Indeed, for each fixed $k$, we observe that $\mu_{t_n}(f_k)-\mu(f_k)$ is the $k$-th entry of the vector $v_{n,t_n}-v_n$ when $n\geq k$. Hence by (\ref{eq:21}),
$$|\mu_{t_n}(f_k)-\mu(f_k)|\leq |v_{n,t_n}-v_n|\leq \frac{1}{n},\quad\quad\ n\geq k.$$
Letting $n\to\infty$, the proof is completed.

We {emphasize} that our choice of $\{t_n\}$ in this step may depend on $\omega\in\Omega_0$.

\bigskip
\noindent{\it Step 3.} We complete our proof of {Theorem} \ref{Thm 2} in this step. That is, we show for any fixed $\omega\in\Omega_0$ in Lemma \ref{lemma cluster} and any $\mu$ in Theorem \ref{Thm 1}, there exists a sequence of times $\{t_n, n\geq1\}$ (that may depend on $\omega$) such that  for all $f\in C^\infty$, $$\mu_{t_n}(f)\to\mu(f).$$ Obviously, since $\mu_t(f)=\mu(f)=0$ for constant $f$, we only need to prove the desired convergence for all $f\in C_0^\infty$.

Let
$$\mathcal{L}=\{f,\  f \mathrm{\ is\ a\ (finite)\ linear\ combination\ of\ } f_1,f_2,...\}.$$
Clearly $\mathcal{L}$ is a dense subset of $H_0^\alpha$ for all $\alpha>0$, for ${ \{\phi_n^\alpha, n\geq1\} }$ is a complete orthonormal system in $H_0^\alpha$.

On the other hand, by what we have proved in {\it Step 2}, together with the linearity of both $\mu_t(\cdot)$ and $\mu(\cdot)$, for each fixed $\omega\in\Omega_0$ there exists $\{t_n, n\geq1\}$ such that for all $f\in\mathcal{L}$,
$$\mu_{t_n}(f)\to \mu(f).$$
By a standard density argument, in order to conclude our proof it suffices to show that, uniformly in $t$, $\mu_t(\cdot)$  is a continuous functional  on the space $H_0^\alpha$ for some $\alpha>0$, and that $\mu(\cdot)$ is a continuous functional on the same $H_0^\alpha$. Fortunately, the desired uniform continuity of $\mu_t(\cdot)$ is given by Theorem \ref{key conti est} for any $\alpha>\max(d-3/2, d/2)$. 

For the continuity of $\mu(\cdot)$, we only need to note that for $f\in H_0^\alpha$,
\begin{align*}
|\mu(f)|&=|(f,g)_{L^2}|\\
&\leq C_1\|f\|_{L^2}\|g\|_{L^2}\\
&\leq C_2\|f\|_{H_0^\alpha}\|g\|_{H_0^1}\\
&\leq 2C_2\sqrt{\frac{1}{m_0}}\|f\|_{H_0^\alpha},
\end{align*}
where we have used the fact that $\|\cdot\|_{L^2}\leq C\|\cdot\|_{H_0^\alpha}$ for $\alpha>0$, and that \[ { \|g\|_{H_0^1}= \sqrt{2}\left\| G^{-1}_{\halffrac}g \right\|_{L^2}\leq 2\sqrt{\frac{1}{m_0}},  }\] by our choice of $\mu$.
The proof of Theorem \ref{Thm 2} is thus completed.\hfill$\Box$


\bibliographystyle{amsplain}

\end{document}